\documentclass[11pt]{article}

\usepackage{amsfonts,amssymb}
\usepackage{diagrams}

\usepackage[latin1]{inputenc}  
\usepackage[T1]{fontenc}

\newarrow{Dots}{}...{}

\def\ac{{\mathcal{AC}}}\def\tc{{\mathcal{TC}}}
\def\TC{{\rm TC}}
\def\ccup{\sqcup}

\makeatletter
\def\revddots{\mathinner{\mkern1mu\raise\p@\vbox{\kern7\p@\hbox{.}}\mkern2mu\raise4\p@\hbox{.}\mkern2mu\raise7\p@\hbox{.}\mkern1mu}}
\makeatother

\newcounter{paragrafsubsub}[subsubsection]
\renewcommand{\theparagrafsubsub}{%
\thesubsubsection.\roman{paragrafsubsub}}
\newcommand{\paragrafsubsub}{%
\refstepcounter{paragrafsubsub}
{\bf \theparagrafsubsub}\hspace{0.2em}--- }

\newcounter{paragrafsub}[subsection]
\renewcommand{\theparagrafsub}{\thesubsection.\arabic{paragrafsub}}
\newcommand{\paragrafsub}{%
\refstepcounter{paragrafsub}
{\bf \theparagrafsub}\hspace{0.2em}--- }

\newcounter{paragraf}[section]
\renewcommand{\theparagraf}{\thesection.\arabic{paragraf}}
\newcommand{\paragraf}{%
\refstepcounter{paragraf}
{\bf \theparagraf}\hspace{0.2em}--- }

\newcommand\paragraphe{%
\par \indent
\ifcase\value{subsection} %
\paragraf
\else
\ifcase\value{subsubsection}\paragrafsub %
\else\paragrafsubsub
\fi\fi
}

\def\longto{\longrightarrow}

\def\lr{{\mathcal LR}}

\def\QQ{{\mathbb Q}}\def\ZZ{{\mathbb Z}}
\def\CC{{\mathbb C}}

\def\Face{{\mathcal F}}

\def\Pic{\rm Pic}

\def\kbprod{{\odot_0}}

\def\Li{{\mathcal{L}}}
\def\Mi{{\mathcal{M}}}

\def\quot{/\hspace{-.5ex}/}

\def\LR{{\rm LR}}

\newtheorem{lemma}{Lemma}

\newtheorem{theo}{Theorem}

\newenvironment{proof}{{\noindent\bf Proof.}}{\hfill $\square$}
\newenvironment{defin}{{\noindent\bf Definition.}}{\\}

\begin{document}
\title{Eigencones and the PRV conjecture}
\author{N. Ressayre\footnote{Universit{\'e} Montpellier II - 
CC 51-Place Eug{\`e}ne Bataillon -
34095 Montpellier Cedex 5 -
France - {\tt ressayre@math.univ-montp2.fr}}}

\maketitle

\begin{abstract}
Let $G$ be a complex semisimple simply connected algebraic group.
Given two irreducible representations $V_1$ and $V_2$ of $G$,
we are interested in some components of $V_1\otimes V_2$.
Consider two geometric realizations of $V_1$ and $V_2$ using the
Borel-Weil-Bott theorem.
Namely, for $i=1,\,2$, let $\Li_i$ be a $G$-linearized line bundle on $G/B$
such that ${\rm H}^{q_i}(G/B,\Li_i)$ is isomorphic to $V_i$.
Assume that the cup product
$$
{\rm H}^{q_1}(G/B,\Li_1)\otimes {\rm H}^{q_2}(G/B,\Li_2)\longto
 {\rm H}^{q_1+q_2}(G/B,\Li_1\otimes\Li_2)
$$
is non zero. Then, ${\rm H}^{q_1+q_2}(G/B,\Li_1\otimes\Li_2)$ is an irreducible
component of $V_1\otimes V_2$; 
such a component is said to be {\it cohomological}.
Solving a Dimitrov-Roth conjecture, we prove here that the cohomological components
of $V_1\otimes V_2$ are exactly the PRV components of stable multiplicity one.
Note that Dimitrov-Roth already obtained some particular cases.
We also characterize these components in terms of the geometry of the Eigencone of $G$.
Along the way, we prove that the structure coefficients of the Belkale-Kumar product
on ${\rm H}^*(G/B,\ZZ)$ in the Schubert basis are zero or one.
\end{abstract}

\section{Introduction}

Let $G$ be a complex semisimple simply connected algebraic group 
with a fixed Borel subgroup $B$ and maximal torus $T\subset B$.
Let $X(T)$ denote the character group of $T$.
For any dominant $\lambda\in X(T)$, $V_\lambda$ denotes the irreducible
$G$-module of highest weight $\lambda$.
We will denote by $\LR(G)$ the set of triples $(\lambda,\,\mu,\,\nu)$ 
of dominant weights such that $V_\lambda\otimes V_\mu\otimes V_\nu$
contains non zero $G$-invariant vectors.
Note that, $(\lambda,\,\mu,\,\nu)$ belongs to $\LR(G)$ if and only if
$V_\nu^*$ is a submodule of $V_\lambda\otimes V_\mu$. 

Let $W$ denote the Weyl group of $T$ and $w_0$ denote the longest element of $W$.
The most obvious component of $V_\lambda\otimes V_\mu$ is $V_{\lambda+\mu}$
corresponding to the point $(\lambda,\,\mu,\,-w_0(\lambda+\mu))$ in $\LR(G)$.
Following Dimitrov-Roth, we present three natural generalizations of these elements
of $\LR(G)$. Our main result which was conjectured and partially proved by Dimitrov-Roth
in \cite{DR:prv1,DR:prv2} asserts that these three generalizations actually coincide.

\bigskip
{\bf The PRV conjecture.}
Let $(\lambda,\,\mu,\,\nu)$ be a triple of dominant weights.
In 1966, Parthasarathy, Ranga-Rao and Varadarajan proved in \cite{PRV} that 
if there exists $w\in W$ such that $w\lambda+ ww_0\mu+w_0\nu=0$ then 
$(\lambda,\,\mu,\,\nu)\in\LR(G)$; and more precisely that 
$(V_{k\lambda}\otimes V_{k\mu}\otimes V_{k\nu})^G$ has dimension one for any
positive integer $k$
(here, $V^G$ denotes the subspace of $G$-invariant vectors in the $G$-module $V$).
Kumar \cite{Kumar:prv1} and Mathieu \cite{Mathieu:prv} independently proved the PRV conjecture
which asserts that $(\lambda,\,\mu,\,\nu)\in\LR(G)$ if there exist $u,v,w\in W$ such that 
$u\lambda+v\mu+w\nu=0$. 
Unlike the original PRV situation, $(V_\lambda\otimes V_\mu\otimes V_\nu)^G$ may have 
dimension greater than one.
Here, we are interested in the set of triple of dominant weights $(\lambda,\,\mu,\,\nu)$
such that:
\begin{enumerate}
\item \label{pr:prv}
$\exists u,v,w\in W$ s.t. $u\lambda+v\mu+w\nu=0$; and,
\item \label{pr:stabone}
$\dim(V_{k\lambda}\otimes V_{k\mu}\otimes V_{k\nu})^G=1$ for any $k\geq 1$.
\end{enumerate}
Such a point in $\LR(G)$ is said to have the PRV property (Property~\ref{pr:prv})
and to have stable multiplicity one (Property~\ref{pr:stabone}).

\bigskip
{\bf Cohomological component of $V_\lambda\otimes V_\nu$.}
Consider the complete flag variety $X=G/B$.
For $\lambda\in X(T)$, we denote by $\Li_\lambda$ the $G$-linearized line bundle on $X$ such 
that $B$ acts on the fiber over $B/B$ by the character $-\lambda$.
If $\lambda$ is dominant the Borel-Weil theorem asserts that ${\rm H}^0(X,\Li_\lambda)$ is
isomorphic to $V_\lambda^*$.
We also set $\lambda^*=-w_0\lambda$.
 The points $(\lambda,\,\mu,\,-w_0(\lambda+\mu))$ of $LR(G)$  have the following 
geometric interpretation:
the morphism 
\begin{eqnarray}
  \label{eq:cupzero}
   {\rm H}^0(X,\Li_\lambda)\otimes  {\rm H}^0(X,\Li_\mu)\longto  
{\rm H}^0(X,\Li_{\lambda+\mu}),
\end{eqnarray}
given by the product of sections is non zero.

Following Dimitrov-Roth (see~\cite{DR:prv1,DR:prv2}), 
we are now going to introduce a natural generalization of 
these points of $\LR(G)$ coming from the Borel-Weil-Bott theorem.
Let $l(w)$ denote the length of $w\in W$ and $\rho$ denote the half sum of the positive
roots. For $w\in W$ and $\lambda\in X(T)$, we set:
\begin{eqnarray}
  \label{eq:Aaff}
  w\cdot\lambda=w(\lambda+\rho)-\rho.
\end{eqnarray}
The Borel-Weil-Bott theorem asserts that for any dominant weight $\lambda$ and
any $w\in W$, ${\rm H}^{l(w)}(X,\Li_{w\cdot\lambda})$ is
isomorphic to $V_\lambda^*$.
Let $(\lambda,\,\mu,\,\nu)$ be a triple of dominant weights.
We will say that $(\lambda,\,\mu,\,\nu^*)$ is a 
{\it cohomological point} of $\LR(G)$ if the cup product:
\begin{eqnarray}
  \label{eq:cup}
   {\rm H}^{l(u)}(X,\Li_{u\cdot\lambda})\otimes  
{\rm H}^{l(v)}(X,\Li_{v\cdot\mu})\longto  
{\rm H}^{l(w)}(X,\Li_{w\cdot\nu}),
\end{eqnarray}
is non zero for some  $u,v,w\in W$ such that $l(w)=l(u)+l(v)$ and
$u\cdot\lambda+v\cdot\mu=w\cdot\nu$. 

\bigskip
{\bf Regularly extremal points.}
Let $\lr(G)$ denote the  cone generated by the semigroup $\LR(G)$ in the 
rational vector space $X(T)_\QQ^3=(X(T)\otimes\QQ)^3$.
Let $X(T)_\QQ^+$ (resp. $X(T)_\QQ^{++}$) denote the cone generated by dominant
(resp. strictly dominant) weights of $T$. 
Since the semigroup $\LR(G)$ is finitely generated, $\lr(G)$ is a closed convex polyhedral 
cone contained in $(X(T)_\QQ^+)^3$.
A face of $\lr(G)$ which intersects $(X(T)_\QQ^{++})^3$ is said to be {\it regular}.
In \cite{GITEigen} and \cite{GITEigen2}, the regular faces are parameterized bijectively. 
In particular, it is proved that the dimension of any regular face is greater 
or equal to $2r$ (where $r$ is the rank of $G$). 
A point in $\LR(G)$ is said to be {\it regularly extremal} if it belongs to a regular face
of $\lr(G)$ of dimension $2r$. 
Note that a regularly extremal point is not necessarily
regular but it is only a limit of regular points in $\lr(G)$ which belongs to 
a minimal regular face.

\bigskip
{\bf The main result.}
We can now state 

\begin{theo}\label{th:main}
Let $(\lambda,\,\mu,\,\nu)$ be a triple of dominant weights.
The following are equivalent:
\begin{enumerate}
\item $(\lambda,\,\mu,\,\nu)$ satisfies the PRV property and has stable multiplicity one;
\item \label{ass:cohom}
$(\lambda,\,\mu,\,\nu)$ is a cohomological point in $\LR(G)$;
\item $(\lambda,\,\mu,\,\nu)$ is regularly extremal.
\end{enumerate}
\end{theo}

Theorem~\ref{th:main} was conjectured in \cite{DR:prv1}.
In \cite{DR:prv2}, Dimitrov-Roth prove it when $\lambda$, $\mu$ or $\nu$ is strictly dominant.
Note that this case also follows easily from \cite[Theorem~G]{GITEigen}.
In \cite{DR:prv2}, Dimitrov-Roth also prove the case when $G$ is a simple classical group.
Here, we present a proof independent of the type of $G$ semisimple.

\bigskip
We now introduce some notation to characterize in a more concrete way the points satisfying
Theorem~\ref{th:main}.
Let $\Phi^+$ denote the set of positive roots. 
For $w\in W$, we consider the following {\it set of inversions of $w$}:
\begin{eqnarray}
  \label{eq:defPhiw}
  \Phi_w:=\{\alpha\in\Phi^+\,:\, -w\alpha\in \Phi^+\}.
\end{eqnarray}
We also set $\Phi^c_w=\Phi^+-\Phi_w$.

We denote by $\ccup$ the disjoint union.
 For example, Condition~(\ref{eq:Phipart}) below means that $\Phi^+$ is the 
disjoint union of $\Phi_u$, $\Phi_v$ and $\Phi_w$.
By \cite[Theorem~I]{DR:prv2}, we have:

\begin{theo}
  \label{th:concret}
Let $(\lambda,\,\mu,\,\nu)$ be a triple of dominant weights.
Then, $(\lambda,\,\mu,\,\nu)$ satisfies (Assertion~(\ref{ass:cohom}) of) 
Theorem~\ref{th:main} if and only if
there exist $u,\,v$ and $w$ in $W$ such that 
\begin{eqnarray}
  \label{eq:Phipart}
  \Phi^+=\Phi_u\ccup\Phi_v\ccup\Phi_w,
\end{eqnarray}
and
\begin{eqnarray}
  \label{eq:mu=0}
  u^{-1}\lambda+v^{-1}\mu+w^{-1}\nu=0.
\end{eqnarray}
\end{theo}
 
\bigskip
{\bf The Belkale-Kumar product for complete flag manifolds.}
Consider the cohomology ring ${\rm H}^*(X,\ZZ)$.
For $w\in W$, we will denote by $\sigma_w$ the cycle class in cohomology
of $\overline{BwB/B}$. The Poincaré dual $\sigma^\vee$ of $\sigma_w$ is
$\sigma_{w_0w}$.
It is well known that ${\rm H}^*(X,\ZZ)=\bigoplus_{w\in W}\sigma_w$.
Along the way, we prove the following:

\begin{theo}
\label{th:bkprod1}
Let   $u,\,v$ and $w$ in $W$ such that $\Phi^+=\Phi_u\ccup\Phi_v\ccup\Phi_w$.
Then, we have:
$$
\sigma_u^\vee\cdot\sigma_v^\vee\cdot\sigma_w^\vee=\sigma_e.
$$
\end{theo}

In \cite{BK}, Belkale-Kumar defined a new product on  ${\rm H}^*(X,\ZZ)$. 
Theorem~\ref{th:bkprod1} actually asserts that the structure coefficients of this product
 in the Schubert basis are zero or one. It allows to compute very easily in this ring.
Note that particular cases of Theorem~\ref{th:bkprod1} were obtained in 
\cite{Richmond:recursion,Rich:mult,multi}. The question to know if Theorem~\ref{th:bkprod1}
holds was explicitly asked in \cite{DR:prv1} and \cite{multi}.

\bigskip
In this paper, we are interested in the question of the existence of non-zero $G$-invariant vectors
in the tensor product of three irreducible $G$-modules.
All the results can be easily generalized to the case of the tensor product of  $s$ such $G$-modules,
for any $s\geq 3$.
\section{GIT cones}

\subsection{Definitions}

Let $X$ be a smooth irreducible projective variety endowed with an algebraic action of $G$.
We assume that the group $\Pic^G(X)$ of $G$-linearized line bundles on $X$
has finite rank. In this work, $X$ will always be  a product of flag manifolds of $G$.
We consider the following semigroup:
\begin{eqnarray}
  \label{eq:TCG}
  \TC^G(X)=\{\Li\in\Pic^G(X)\;:\;{\rm H}^0(X,\Li)^G\neq\{0\}\}.
\end{eqnarray}
The Borel-Weil theorem allows to identify $\TC^G((G/B)^3)$ with $\LR(G)$.
We will denote by $\tc^G(X)$ the cone generated by $\TC^G(X)$ in 
$\Pic^G(X)_\QQ=\Pic^G(X)\otimes\QQ$.
The set of ample $G$-linearized line bundles on $X$ generates an open convex cone 
$\Pic^G(X)^{++}_\QQ$   in $\Pic^G(X)_\QQ$. We set:
\begin{eqnarray}
  \label{eq:lrG}
  \ac^G(X)=\Pic^G(X)^{++}_\QQ\cap\tc^G(X).
\end{eqnarray}
For example, $\tc^G((G/B)^3)$ is $\lr(G)$ and  $\ac^G((G/B)^3)$ is the intersection 
of $\lr(G)$ with the interior of the dominant chamber of $X(T^3)_\QQ$. 

\bigskip
For any $\Li\in\Pic^G(X)$, 
we set 
$$
X^{\rm ss}(\Li)=\{x\in X \,:\,\exists n>0{\rm\ and\ }\sigma\in{\rm H}^0(X,\Li^{\otimes n})^G 
\  {\rm s.\,t.}\ \sigma(x)\neq 0\}.
$$
Note that this definition of $X^{\rm ss}(\Li)$ is like in \cite{GIT} if $\Li$ is ample
but not in general.
We consider the following projective variety:
$$
X^{\rm ss}(\Li)\quot G:={\rm Proj}\bigoplus_{n\geq 0}{\rm H}^0(X,\Li^{\otimes n})^G,
$$
and the natural $G$-invariant morphism
$$
\pi\,:\,X^{\rm ss}(\Li)\longto X^{\rm ss}(\Li)\quot G.
$$
If $\Li$ is ample $\pi$ is a good quotient.

\subsection{Covering pairs}

\paragraphe
Let $P$ be a parabolic subgroup of $G$ containing $B$.
Le $W_P$ denote the Weyl group of $P$ and
 $W^P$ denote the set of minimal length representatives of elements in $W/W_P$.
For $u\in W^P$, we will denote by $\sigma_u^P$ the cycle class in
 ${\rm H}^{2\dim(G/P)-2l(u)}(G/P,\ZZ)$ of $\overline{BuP/P}$.
Let us consider the tangent space $T_u$ of
$u^{-1}BuP/P$ at the point $P$. 

Using Kleiman's transversality theorem, Belkale-Kumar showed in \cite[Proposition~2]{BK} the 
following important lemma:

\begin{lemma}\label{lem:fondBK}
Let now $u,\,v$ and $w$ in $W^P$ such that $l(u)+l(v)+l(w)=\dim G/P$.
  The product $\sigma^P_u\cdot\sigma^P_v\cdot\sigma^P_w$ is non zero if and only if 
there exist $p_1,p_2,p_3\in P$ such that 
the natural map
$$
T_P(G/P)\longto \frac{T_P(G/P)}{p_1T_u}\oplus \frac{T_P(G/P)}{p_2T_v}\oplus \frac{T_P(G/P)}{p_3T_w},
$$
is an isomorphism.
\end{lemma}

Then, Belkale-Kumar defined Levi-movability:\\

\begin{defin}
The triple $(u,\,v,\,w)$ is said to be {\it Levi-movable} if there exist
 $l_1,l_2,l_3\in L$ such that the natural map
$$
T_P(G/P)\longto \frac{T_P(G/P)}{l_1T_u}\oplus \frac{T_P(G/P)}{l_2T_v}\oplus \frac{T_P(G/P)}{l_3T_w},
$$
is an isomorphism.
\end{defin}

We define $c_{uvw}\in\ZZ_{\geq 0}$ by
$$
\sigma^P_{u}.\sigma^P_{v}=\sum_{w\in W^P} c_{uvw}(\sigma^P_w)^\vee,
$$
where $(\sigma^P_w)^\vee$ denotes the Poincaré dual class of $\sigma^P_w$.
Belkale-Kumar set:
$$
c_{uvw}^\kbprod=\left\{
  \begin{array}{ll}
c_{uvw}&{\rm \ if\ }  (u,\,v,\,w) {\rm\ is\ Levi-movable;}\\
    0&{\rm\ otherwise.}
  \end{array}
\right .
$$

They define on the group ${\rm H}^*(G/P,\ZZ)$ a bilinear product $\kbprod$ by the formula:
$$
\sigma^P_{u}\kbprod\sigma^P_{v}=\sum_{w\in W^P} c_{uvw}^\kbprod(\sigma^P_w)^\vee.
$$
By \cite[Definition~18]{BK}, we have:

\begin{theo}
  The product $\kbprod$ is commutative, associative and satisfies Poincar\'e duality.
\end{theo}

Note that $T_u$ is stable by $T$. This implies that for $P=B$, $(u,\,v\,\,w)$ is 
Levi-movable if and only if 
\begin{eqnarray}
  \label{eq:levimovB1}
l(u)+l(v)+l(w)=2l(w_0),\ {\rm and}\\
  T_u\cap T_v\cap T_w=\{0\}.
\end{eqnarray}
Since the weights of $T$ in $T_u$ are precisely $-\Phi_u^c$, one can easily checks that
$(u,\,v\,\,w)$ is 
Levi-movable if and only if 
\begin{eqnarray}
  \label{eq:levimovB}
\Phi_u^c\ccup\Phi_v^c\ccup\Phi_w^c=\Phi^+.
\end{eqnarray}

\paragraphe
Let $H$ be a subtorus of $T$ and $C$ be an irreducible subvariety of the $H$-fixed point set
$X^H$ in $X$.
Let $\Li\in \Pic^G(X)$. There exists a unique character $\mu^\Li(C,H)$ of $H$ such that
\begin{eqnarray}
  \label{eq:muLH}
  h.\tilde{x}=\mu^\Li(C,H)(h^{-1})\tilde{x},
\end{eqnarray}
for any $h\in H$ and $\tilde{x}\in\Li$ over $C$.
Analogously, if $\lambda$ is a one parameter subgroup of $G$ and $C$ is an irreducible 
subvariety of $X^\lambda=X^{{\rm Im}\lambda}$, we will denote by $\mu^\Li(C,\lambda)$
the integer such that:
\begin{eqnarray}
  \label{eq:mulambda}
  \lambda(t)\tilde{x}=t^{-\mu^\Li(C,\lambda)}\tilde{x},
\end{eqnarray}
for all $t\in\CC^*$ and $\tilde{x}$ as above.

\bigskip
We will consider the  parabolic subgroup $P(\lambda)$ (see \cite{GIT}) defined by
\begin{eqnarray}
  \label{eq:Plambda}
  P(\lambda)=\{g\in G\;:\;\lim_{t\to 0}\lambda(t)g\lambda(t^{-1}) {\rm\ exists\ in\ }G\}.
\end{eqnarray}
We also denote by $G^\lambda$ the centralizer of $\lambda$ in $G$; it is a Levi subgroup
of $P(\lambda)$.
Now, $C$ is an irreducible component of $X^\lambda$. We denote by $C^+$ the
corresponding  Bia\l{}ynicki-Birula cell:
\begin{eqnarray}
  \label{eq:BiBi}
  C^+=\{x\in X\;:\;\lim_{t\to 0}\lambda(t)x\in C\}.
\end{eqnarray}
One can easily check that $C^+$ is $P(\lambda)$-stable. 
We consider the fiber product $G\times_{P(\lambda)}C^+$ and the morphism
\begin{eqnarray}
  \label{eq:eta}
  \begin{array}{cccc}
    \eta\;:&G\times_{P(\lambda)}C^+&\longto&X\\
&[g:x]&\longmapsto&g.x.
  \end{array}
\end{eqnarray}

\begin{defin}
  The pair $(C,\lambda)$ is said to be {\it generically finite} if 
$\eta$ is dominant with finite general fibers.
It is said to be {\it well generically finite} if 
it is generically finite and there exists a point $x\in C$ such that 
the tangent map of $\eta$ at $[e:x]$ is invertible.
It is said to be {\it well covering} if it is well generically finite and $\eta$ is birational.
\end{defin}

\paragraphe
Set $X=(G/B)^3$.
Let $\lambda$ be a dominant regular one parameter subgroup; $P(\lambda)=B$.
The group $\lambda$ has only isolated fixed points in $X$ parameterized by $W^3$.
Let $(u,\,v,\,w)\in W$ and $z=(u^{-1}B,\,v^{-1}B,\,w^{-1}B)\in X$.
Set $C=\{z\}$.
It is well known that $C^+=Bu^{-1}B\times Bv^{-1}B\times Bw^{-1}B\subset X$.
Consider now 
$$
\eta\;:\;G\times_BC^+ \longto X.
$$
Let $x=(g_1B,\,g_2B,\,g_3B)\in X$.
The projection $G\times_BC^+ \longto G/B$ induces an isomorphism between 
$\eta^{-1}(x)$ and $g_1BuB\cap g_2BvB\cap g_3BwB$.
With the Kleiman theorem, this implies that 

\begin{center}
\begin{tabular}{r@{\,}c@{\,}l}
  $(C,\lambda)$ is generically finite&$\iff$&
$\sigma_u.\sigma_v.\sigma_w=d[{\rm pt}]$ with $d>0$;\\
$(C,\lambda)$ is well generically finite&$\iff$&
$\sigma_u.\sigma_v.\sigma_w=d[{\rm pt}]$ with $d>0$, and\\
&&$\eta^{-1}(z)$ is finite;\\
$(C,\lambda)$ is well covering &$\iff$&
$\sigma_u.\sigma_v.\sigma_w=[{\rm pt}]$, and\\
&&$\eta^{-1}(z)=\{[e:z]\}$.\\

\end{tabular}
\end{center}
\subsection{PRV points in $\LR(G)$}

In this subsection, $X=(G/B)^3$. We have the following very easy lemma:

\begin{lemma}
\label{lem:prvC}
Let $(\lambda,\,\mu,\,\nu)$ be a triple of dominant weights.
Then, $(\lambda,\,\mu,\,\nu)$ has the PRV property if and only if there exists
an irreducible component $C$ of $X^T$ such that $\mu^{\Li_{(\lambda,\,\mu,\,\nu)}}(C,T)$ 
is trivial.
\end{lemma}

\begin{proof}
  The irreducible components of $X^T$ are the singletons $\{(uB,\,vB,\,wB)\}$ for
$u,\,v,\,w\in W$. Moreover, a direct computation shows that 
$$\mu^{\Li_{(\lambda,\,\mu,\,\nu)}}(\{(uB,\,vB,\,wB)\},T)=
-(u\lambda+v\mu+w\nu).$$
The lemma follows.
\end{proof}

\bigskip
We also make the following obvious observation:

\begin{lemma}
\label{lem:stabone}
  Let $(\lambda,\,\mu,\,\nu)$ be a point in $\LR(G)$ with the PRV property.
Then, $(\lambda,\,\mu,\,\nu)$ has stable multiplicity one if and only if 
$X^{\rm ss}(\Li_{(\lambda,\,\mu,\,\nu)})\quot G$ is a point.
\end{lemma}

\subsection{Cohomological points in $\LR(G)$}

We now recall \cite[Theorem~1]{DR:prv2}:

\begin{theo}\label{th:DR}
  Let $(\lambda,\,\mu,\,\nu)$ be a triple of dominant weights.
Then, $(\lambda,\,\mu,\,\nu)$ is a cohomological point of $\LR(G)$ if and only if 
there exist $u,\,v,\,w\in W$ such that 
\begin{enumerate}
\item $u^{-1}\lambda+v^{-1}\mu+w^{-1}\nu=0$, and
\item  $\Phi^+=\Phi_u\ccup\Phi_v\ccup\Phi_w$.
\end{enumerate}
\end{theo}

%Let us a regular dominant one parameter subgroup $\lambda$. Then, $P(\lambda)=B$.

\subsection{Regularly extremal points in $\LR(G)$}

We now recall a result from \cite{GITEigen,GITEigen2} which describes the regularly extremal
points in $\LR(G)$. Indeed, in \cite{GITEigen,GITEigen2}, we describe the minimal regular 
faces of $\lr(G)$, and the Kumar-Mathieu version of the PRV conjecture proves that $\LR(G)$ 
is saturated along these faces (that is, any triple of dominant weights which belongs to
$\lr(G)$ belongs to $LR(G)$).

\begin{theo}
  \label{th:regext}
Let  $(\lambda,\,\mu,\,\nu)$ be a triple of dominant weights.
Then, $(\lambda,\,\mu,\,\nu)$ is a regularly extremal point of $\LR(G)$ if and only if 
there exist $u,\,v,\,w\in W$ such that
\begin{enumerate}
\item $u^{-1}\lambda+v^{-1}\mu+w^{-1}\nu=0$,
\item $\Phi^+=\Phi_u^c\ccup\Phi_v^c\ccup\Phi_w^c$, and
\item $\sigma_u\cdot\sigma_v\cdot\sigma_w=\sigma_e$.
\end{enumerate}
\end{theo}

\section{The Belkale-Kumar product for complete flag manifolds}

\begin{theo}\label{th:bkprod}
  The non-zero structure coefficients  of the ring $({\rm H}^*(G/B,\ZZ),\kbprod)$
in the  Schubert basis are equal to $1$.
\end{theo}

\begin{proof}
  Let $(u,\,v,\,w)\in W^3$ such that 
$$
\sigma_{u}\kbprod\sigma_{v}\kbprod\sigma_{w}=d[{\rm pt}].
$$
Note that $d$ is the coefficient of $\sigma_{w}^\vee$ in the expression
of $\sigma_{u}\kbprod\sigma_{v}$ as a linear combination of 
Schubert classes.
So, we have to prove that if $d\neq 0$  then $d=1$.

Set $X=(G/B)^3$, $z=(u^{-1}B,\,v^{-1}B,\,w^{-1}B)$ and 
$C^+=Bu^{-1}B\times Bv^{-1}B\times Bw^{-1}B)$. 
Consider the following morphism
$$
\begin{array}{lccc}
  \eta\,:&G\times_B C^+&\longto&X\\
&[g:x]&\longmapsto&gx.
\end{array}
$$
Since $(u,\,v,\,w)$ is Levi-movable, the tangent map of $\eta$ is invertible at $[e:z]$
and so at any point of $C^+$.
It follows that $\eta$ is a covering of degree $d$.
In particular $d$ is the cardinality of the fiber $\eta^{-1}(z)$.

Consider the natural projection $\pi\,:\,G\times_B C^+\longto G/B$.
Choose a one parameter subgroup $\lambda$ of $T$ such that $P(\lambda)=B$; that is, 
$\lambda$ is  dominant and regular.
The map $\pi$ identifies $\eta^{-1}(z)$ with
the set of $gB\in G/B$ such that $g^{-1}z\in C^+$.
Since $\lim_{t\to 0}\lambda(t)(g^{-1}z)=z$, \cite[Lemma~12]{GITEigen} implies that $g^{-1}z\in Bz$.
So, $g\in G_zB$. Finally,  $\eta^{-1}(z)=G_z.B\subset G/B$.

It remains to prove that $G_z$ is connected. 
Let $g\in G_z$. Since $T$ and $gTg^{-1}$ are maximal tori of $G_z^\circ$, there exists
$h\in G_z^\circ$ such that $gTg^{-1}=hTh^{-1}$.
Then, $h^{-1}g$ normalizes $T$. But, $h^{-1}g$ fixes $u^{-1}B$. 
We deduce that $h^{-1}g$ belongs to $T$ and so to $G_z^\circ$. 
It follows that $g$ belongs to $G_z^\circ$.
\end{proof}

\section{The main theorem}

Lemmas~\ref{lem:prvC} and \ref{lem:stabone}, Theorems~\ref{th:DR} and \ref{th:regext} show 
that Theorem~\ref{th:main} is equivalent to the following

\begin{theo}
\label{th:main2}
  Let  $(\lambda,\,\mu,\,\nu)$ be a triple of dominant weights.
Then, the following are equivalent
\begin{enumerate}
\item \label{ass:prv}
 there exist $u,\,v,\,w\in W$ such that
\begin{enumerate}
\item $u^{-1}\lambda+v^{-1}\mu+w^{-1}\nu=0$, and
\item $X^{\rm ss}(\Li_{(\lambda,\,\mu,\,\nu)})\quot G$ is a point.
  \end{enumerate}
\item \label{ass:coh}
there exist $u,\,v,\,w\in W$ such that
\begin{enumerate}
\item $u^{-1}\lambda+v^{-1}\mu+w^{-1}\nu=0$, and
\item $\Phi^+=\Phi_u\ccup\Phi_v\ccup\Phi_w$,
\end{enumerate}
\item \label{ass:ext}
there exist $u,\,v,\,w\in W$ such that
\begin{enumerate}
\item $u^{-1}\lambda+v^{-1}\mu+w^{-1}\nu=0$,
\item $\Phi^+=\Phi_u^c\ccup\Phi_v^c\ccup\Phi_w^c$, and
\item $\sigma_u\cdot\sigma_v\cdot\sigma_w=\sigma_e$.
\end{enumerate}
\end{enumerate}
\end{theo}

We first prove

\begin{lemma}
  \label{lem:Xpt}
Let $G$ be a reductive group and $Y$ be a product of flag varieties of $G$.
We assume that $\ac^G(Y)=\Pic^G(Y)_\QQ^{++}$.

Then, $Y$ is a point.
\end{lemma}

\begin{proof}
We are going to prove that if $Y$ is not a point, then $\ac^G(Y)$ is not equal to 
$\Pic^G(Y)_\QQ^{++}$.
If $Y=G/P_1$ with $P_1$ a strict parabolic subgroup of $G$, $\ac^G(Y)$ is empty.
If $Y=G/P_1\times G/P_2$ with $P_1$ and $P_2$ two strict parabolic subgroups of $G$,
a weight $(\lambda,\mu)$ belongs to $\ac^G(Y)$ if and only if $\mu=-w_0\lambda$.
In particular, $\ac^G(Y)$ has empty interior.

Let us now assume that, $Y=G/P_1\times G/P_2\times G/P_3$ with $P_1$, $P_2$ and $P_3$ 
three strict parabolic subgroups of $G$.
Let $(\lambda,\,\mu,\,\nu)$ be three weights such that $\Li_{(\lambda,\,\mu,\,\nu)}$
is an ample line bundle on $Y$.
The set of $\gamma\in X(T)\otimes\QQ$ such that
there exists a positive integer $k$ such that 
$V_{k\gamma}^*$ is contained in $V_{k\lambda}\otimes V_{k\mu}$ is a compact polytope
(namely, a moment polytope). In particular, there exists $n$ 
such that for any positive integer $k$, $V_{kn\nu}^*$ is not a submodule of 
$V_{k\lambda}\otimes V_{k\mu}$.
So, the ample element $\Li_{(\lambda,\,\mu,\,n\nu)}$ does not belong to $\ac^G(Y)$.

The case when $Y$ is a product of more than three flag varieties works similarly.
\end{proof}

\bigskip
\begin{proof}[of Theorem~\ref{th:main2}]
Note that for any $u\in W$, we have
$$
\begin{array}{r@{\,}l}
\Phi_{w_0u}&=\{\alpha\in\Phi^+\,|\,-w_0u\alpha\in\Phi^+\}\\&=
\{\alpha\in\Phi^+\,|\,u\alpha\in\Phi^+\}\\&=\Phi_u^c;
\end{array}$$
and 
$$
\begin{array}{r@{\,}l}
  \Phi_{uw_0}&=\{\alpha\in\Phi^+\,|\,u(-w_0\alpha)\in\Phi^+\}\\&=
-w_0\{\alpha\in\Phi^+\,|\,u\alpha\in\Phi^+\}\\&=-w_0\Phi_u^c.
\end{array}
$$

Assume that Assertion~\ref{ass:ext} is satisfied for $u,\,v$ and $w$ in $W$.
Then, we have:
$$
\begin{array}{r@{}l}
  \Phi^+=-w_0\Phi^+&=(-w_0\Phi_u^c)\ccup(-w_0\Phi_v^c)\ccup(-w_0\Phi_w^c)\\
&=\Phi_{uw_0}\ccup\Phi_{vw_0}\ccup\Phi_{ww_0}.
\end{array}
$$
So, $uw_0,\,vw_0$ and $ww_0$ satisfy Assertion~\ref{ass:coh}.\\

Conversely, assume that  Assertion~\ref{ass:coh} is satisfied for $u',\,v'$ and $w'$ in $W$.
Set $u=u'w_0,\,v=v'w_0$ and $w=w'w_0$. 
The above proof shows that $\Phi^+=\Phi_u^c\ccup\Phi_v^c\ccup\Phi_w^c$.
Theorem~\ref{th:bkprod} shows that $\sigma_u.\sigma_v.\sigma_w=\sigma_e$.
The identity $u^{-1}\lambda+v^{-1}\mu+w^{-1}\nu=0$ follows from
$u'^{-1}\lambda+v'^{-1}\mu+w'^{-1}\nu=0$.
Finally, Assertion~\ref{ass:ext} holds.\\ 

Let us assume that Assertion~\ref{ass:ext} is satisfied and set 
$C=\{(u^{-1}B,\,v^{-1}B,\,w^{-1}B)\}$. By \cite[Proposition~9]{GITEigen}, there exists
a dominant morphism from $C$ to   $X^{\rm ss}(\Li_{(\lambda,\,\mu,\,\nu)})\quot G$; 
it follows that  $X^{\rm ss}(\Li_{(\lambda,\,\mu,\,\nu)})\quot G$ is a point.\\

Let us assume that $(\lambda,\,\mu,\,\nu)$ satisfies Assertion~\ref{ass:prv}.
If $\overline{X}$ is a  product of three flag manifolds for $G$, there exists a
unique $G^3$-equivariant map $p\,:\,X\longto \overline{X}$.
There exists a unique such variety $\overline{X}$, such that 
$\Li_{(\lambda,\,\mu,\,\nu)}$ is the pullback by $p$ of an ample $G$-linearized line bundle
$\overline{\Li}$ on $\overline{X}$. 
Consider the image $\overline{z}$ of $(u^{-1}B,\,v^{-1}B,\,w^{-1}B)$ by $p$.

\bigskip
The condition $u^{-1}\lambda+v^{-1}\mu+w^{-1}\nu=0$ implies that $T$ acts trivially 
on the fiber in $\overline{\Li}$ over $\overline{z}$.
Since $T$ has finite index in its normalizer $N(T)$  in $G$, $\overline{z}$ 
is semitable for $\overline{\Li}$ and the action of $N(T)$.
A Luna theorem (see \cite[Proposition~8]{GITEigen} for an adapted version) shows that 
$\overline{z}$ is semistable for $\overline{\Li}$ and the action of $G$.
In particular, $\overline{\Li}$ belongs to $\ac^G(\overline{X})$.

Let $\overline{\Face}$ be the face of $\tc^G(\overline{X})$ containing 
$\overline{\Li}$ in its relative interior.
By \cite[Theorem~H]{GITEigen}, there exists a well covering pair
$(\overline{C},\lambda)$ of $\overline{X}$ such that
$\overline{\Face}$ is the set of $\Li\in\tc^G(\overline{X})$
such that $\mu^\Li(\overline{C},\lambda)=0$.
The first step of this proof is to show that there exists such a pair where $\overline{C}$
is a singleton.

\bigskip
 By \cite{carquois}, there exists a well covering pair
$(\overline{C},\lambda)$ of $\overline{X}$ such that
\begin{enumerate}
\item $\lambda$ is a dominant one parameter subgroup of $T$;
\item $\overline{\Face}$ is the set of $\Li\in\tc^G(\overline{X})_\QQ$ 
such that $\mu^\Li(\overline{C},\lambda)=0$;
\item $\overline{\Li}_{|\overline{C}}$ belongs to the relative interior of 
$\ac^{G^\lambda}(\overline{C})$;
\item \label{cond:acrond}
if $K$ is the kernel of the action of $G^\lambda$ on $\overline{C}$,
$\ac^{G^\lambda}(\overline{C})$ spans the subspace $\Pic^G(\overline{X})_\QQ^K$.
\end{enumerate}
We claim that $\overline{C}$ is a singleton.
We mention that the proof of the claim will use Lemma~\ref{lem:Xpt}.

We first prove that $G.\overline{z}$ is the unique closed $G$-orbit in 
$\overline{X}^{\rm ss}(\overline{\Li})$.
Since $\overline{X}^{\rm ss}(\overline{\Li})\quot G=
X^{\rm ss}(\Li_{(\lambda,\,\mu,\,\nu)})\quot G$ is a point,
$\overline{X}^{\rm ss}(\overline{\Li})$ is affine and contains a unique closed $G$-orbit.
Since $\overline{z}$ is fixed by $T$ and $B/T$ is unipotent, $B.\overline{z}$ is closed in the affine
variety  $\overline{X}^{\rm ss}(\overline{\Li})$.
Since $G/B$ is complete, we deduce  that $G.\overline{z}$ is closed 
in $\overline{X}^{\rm ss}(\overline{\Li})$.\\

By \cite[Proposition~10]{GITEigen}, $\overline{C}$ intersects $G.\overline{z}$.
Up to changing $\overline{z}$ by another point in $W.\overline{z}$, one
may assume that $\overline{z}\in\overline{C}$.\\

We claim that  $\ac^{G^\lambda}(\overline{C})$ is the set of points in
$\Pic^{G^\lambda}(\overline{C})_\QQ^{++}$ with trivial action of $K^\circ$.
By Condition~\ref{cond:acrond}, it is sufficient to prove that 
$\ac^{G^\lambda}(\overline{C})$ is the intersection of
$\Pic^{G^\lambda}(\overline{C})_\QQ^{++}$ and a linear subspace.
The kernel of $\mu^\bullet(\overline{z},T)$ will be this subspace.
By \cite[Corollary~1]{Luna:adh} (see also, \cite[Proposition~8]{GITEigen}), 
if $\Mi\in \Pic^{G^\lambda}(\overline{C})_\QQ^{++}$
satisfy $\mu^\Mi(\overline{z},T)=0$ then $\overline{z}$ is semistable for
$\Mi$ and $\Mi$ belongs to $\ac^{G^\lambda}(\overline{C})$.
Since $\overline{C}^{\rm ss}(\overline{\Li}_{|\overline{C}})\quot G^\lambda$ is a point,
$G^\lambda\overline{z}$ is the unique closed $G^\lambda$-orbit in 
$\overline{C}^{\rm ss}(\overline{\Li}_{|\overline{C}})$.
But, $\overline{\Li}_{|\overline{C}}$ belongs to the relative interior of 
$\ac^G(\overline{C})$. It follows that $G^\lambda.\overline{z}$ is the only closed $G^\lambda$-orbit
in $\overline{C}^{\rm ss}(\Mi)$ for any $\Mi$ in the relative interior of 
$\ac^G(\overline{C})$. In particular, $\mu^\Mi(\overline{z},T)=0$.
This implies that $\ac^{G^\lambda}(\overline{C})$ is contained in 
the kernel of $\mu^\bullet(\overline{z},T)$.

\bigskip
The claim and Lemma~\ref{lem:Xpt} below imply that $\overline{C}$ is  one point; so,
$\overline{C}=\{\overline{z}\}$. This ends the first step.

\bigskip
The second step consists in proving that $G_{\overline{z}}=G^\lambda$.
Consider $\overline{\eta}\,:\,G\times_{P(\lambda)}\overline{C}^+\longto\overline{X}$.
Since $(\overline{C},\lambda)$ is well covering, $\overline{\eta}^{-1}(z)$ is only one point.
This implies that $G_{\overline{z}}$ is contained in $P(\lambda)$.
On the other hand, $G^\lambda$ is connected and acts on each irreducible component of 
$\overline{X}^\lambda$. We deduce that $G^\lambda$ fixes $\overline{z}$.
Moreover, $G.\overline{z}$ is affine, and $G_{\overline{z}}$ is reductive.
This implies that $G_{\overline{z}}=G^\lambda$.

\bigskip
The third step consists in raising $(\overline{C},\lambda)$ to a well covering pair 
$(C,\lambda)$ of $X$.
Let $P$, $Q$ and $R$ be the parabolic subgroups of $G$ containing $B$ such that 
$\overline{X}=G/P\times G/Q\times G/R$. 
Up to multiplying $u$ by an element of $W_P$ on the left, we may assume that
 $\overline{BuP(\lambda)}=\overline{PuP(\lambda)}$.
Similarly, we choose $v$ and $w$ without changing 
$\overline{z}=(u^{-1}P,\,v^{-1}Q,\,w^{-1}R)$.
Since $(\overline{C},\lambda)$ is well covering, \cite[Proposition~11]{GITEigen} shows that:
\begin{eqnarray}
  \label{eq:1}
  [\overline{BuP(\lambda)}]\kbprod
[\overline{BvP(\lambda)}]\kbprod
[\overline{BwP(\lambda)}]=[{\rm pt}]\in {\rm H}^*(G/P(\lambda),\ZZ).
\end{eqnarray}
Set $C=G^\lambda u^{-1}B\times G^\lambda v^{-1}B\times G^\lambda w^{-1}B\subset G/B^3$.
Then, \cite[Proposition~11]{GITEigen} shows that $(C,\lambda)$ 
is a well covering pair of $X$. The corresponding face $\Face$ of $\lr(G)$ contains $\overline{\Face}$.

\bigskip
The forth step consists in perturbing $(C,\lambda)$ to obtain a well covering pair 
$(C',\lambda')$ with a {\bf regular} one parameter subgroup $\lambda'$
such that the corresponding face $\Face'$ of $\lr(G)$ still contains $\overline{\Face}$.
Let us recall that the map $W^{P(\lambda)}\times W_{P(\lambda)}\longto W$, $(u,v)\mapsto uv$ is 
a bijection. For $w\in W$, we will denote by $\bar w$ the unique element of $W_{P(\lambda)}$
such that $w\in W^{P(\lambda)}w$.
Since $G_{\overline{z}}=G^\lambda$, one can multiply  $u$, $v$ and $w$ on the right by
 elements of $W_{G^\lambda}$ to obtain:
 \begin{enumerate}
 \item $\overline{z}=(u^{-1}P,\,v^{-1}Q,\,w^{-1}R)$,
\item $[\overline{BuP(\lambda)}]\kbprod
[\overline{BvP(\lambda)}]\kbprod
[\overline{BwP(\lambda)}]=
[{\rm pt}]\in{\rm H}^0(G/P(\lambda),\ZZ)$,
\item \label{condbar}
$[\overline{B^\lambda \bar uB^\lambda}]\kbprod
[\overline{B^\lambda \bar vB^\lambda}]\kbprod
[\overline{B^\lambda \bar wB^\lambda}]=
[{\rm pt}]\in{\rm H}^0(G^\lambda/B^\lambda,\ZZ)$.
 \end{enumerate}
We claim that 
\begin{eqnarray}
  \label{eq:2}
  [\overline{BuB}]\kbprod
[\overline{BvB}]\kbprod
[\overline{BwB}]=[{\rm pt}]\in{\rm H}^0(G/B,\ZZ).
\end{eqnarray}
Set $z=(u^{-1}B,\,v^{-1}B,\,w^{-1}B)$ and
$C^+=B^3.z$.
Consider the morphism $\eta\,:\,G\times_BC^+\longto X$.
To prove the claim, we have to prove that $\eta$ is birational 
and that $\eta^{-1}(z)=\{[e:z]\}$.
By \cite{multi} or \cite{Rich:mult}, $[\overline{BuB}]\cdot
[\overline{BvB}]\cdot
[\overline{BwB}]=[{\rm pt}]$ and $\eta$ is birational.
Let now $g\in G$ such that $g^{-1}z\in C^+$. 
It remains to prove that $g\in B$.
Since $C^+=B^3.z$ and $\overline{C}^+=P(\lambda)^3p(z)$, 
$g^{-1}\overline{z}\in\overline{C}^+$. 
But, $(\overline{C}^+,\lambda)$ is well covering, and so,
$g\in P(\lambda)$.
Since $P(\lambda)=G^\lambda B$, we may assume that $g\in G^\lambda$.

Consider now, the subvariety
 $F=G^\lambda u^{-1}B\times G^\lambda v^{-1}B\times G^\lambda w^{-1}B$ of $X$.
There is a unique $G^\lambda$-equivariant isomorphism from $F$ onto 
$(G^\lambda/B^\lambda)^3$ and $C^+\cap F$ maps onto
$B^\lambda \bar u^{-1}B^\lambda\times
B^\lambda \bar v^{-1}B^\lambda\times
B^\lambda \bar w^{-1}B^\lambda$
by this isomorphism. 
Now, since $g^{-1}z\in C^+\cap F$ and $g\in G^\lambda$, 
Condition \ref{condbar} implies that $g\in B^\lambda$.

Finally, Condition~\ref{eq:2} means that $(u,v,w)$ satisfies Assertion~\ref{ass:ext}.
\end{proof}

\bibliographystyle{amsalpha}
\bibliography{biblio}

\begin{center}
  -\hspace{1em}$\diamondsuit$\hspace{1em}-
\end{center}
\end{document}